\documentclass[12pt,a4paper]{article}
\usepackage{amssymb,amsmath,amsthm}
\usepackage{graphicx}

\newcommand\cF{{\mathcal F}}

\newcommand\cG{{\mathcal G}}

\newcommand\cS{{\mathcal S}}

\theoremstyle{plain}
\newtheorem{theorem}{Theorem}[section]
\newtheorem{lemma}[theorem]{Lemma}

\theoremstyle{definition}

\newcommand\cref[1]{Corollary~\ref{cor:#1}}

\title{On the general position problem on Kneser graphs}
\author{Bal\'azs Patk\'os
\medskip \\
\small Alfr\'ed R\'enyi Institute of Mathematics, Hungarian Academy of Sciences\\
\small P.O.B. 127, Budapest H-1364, Hungary. \\
\medskip
\small \texttt{
patkos@renyi.hu}}

\begin{document}

\maketitle

\begin{abstract}
    In a graph $G$, a \textit{geodesic} between two vertices $x$ and $y$ is a shortest path connecting $x$ to $y$. A subset $S$ of the vertices of $G$ is \textit{in general position} if no vertex of $S$ lies on any geodesic between two other vertices of $S$. The size of a largest set of vertices in general position is the \textit{general position number} that we denote by $gp(G)$. Recently, Ghorbani et al, proved that for any $k$ if $n\ge k^3-k^2+2k-2$, then $gp(Kn_{n,k})=\binom{n-1}{k-1}$, where $Kn_{n,k}$ denotes the Kneser graph. We improve on their result and show that the same conclusion holds for $n\ge 2.5k-0.5$ and this bound is best possible. Our main tools are a result on cross-intersecting families and a slight generalization of Bollob\'as's inequality on intersecting set pair systems.
\end{abstract}

\section{Introduction}
A recently studied extremal problem \cite{india,ize,MK} in graph theory is the following: in a graph $G$, a \textit{geodesic} between two vertices $x$ and $y$ is a shortest path connecting $x$ to $y$. We say that a subset $S$ of the vertices of $G$ is  \textit{in general position} if no vertex of $S$ lies on any geodesic between two other vertices of $S$. The size of a largest set of vertices in general position is the \textit{general position number} which we denote by $gp(G)$. Our graph of interest in this paper is the \textit{Kneser graph} $Kn_{n,k}$ whose vertex is $\binom{[n]}{k}$, the set of all $k$-element subsets of the set $[n]=\{1,2,\dots,n\}$ and two $k$-subsets $S$ and $T$ are joined by an edge if and only if $S\cap T=\emptyset$.
Ghorbani et al \cite{szloven} determined $gp(Kn_{n,2})$ and $gp(Kn_{n,3})$ for all $n$ and showed that for any fixed $k$ if $n$ is large enough, then $gp(Kn_{n,k})=\binom{n-1}{k-1}$ holds.

\begin{theorem}[\cite{szloven}]\label{szlov}
Let $n, k\ge 2$ be integers with $n\ge 3k-1$.  If for all $t$,  where $2\le t\le k$,  the inequality $k^t\binom{n-t}{k-t}+t\le \binom{n-1}{k-1}$ holds, then $gp(Kn_{n,k}) =\binom{n-1}{k-1}$.
\end{theorem}

For fixed $k$ and $t=2$ the above inequality is satisfied when $n\ge k^3-k^2+2k-1$ holds. We improve on this and the main result of this note is the following.

\begin{theorem}\label{kneser}
If $n, k\ge 4$ are integers with $n\ge 2k+1$, then $gp(Kn_{n,k}) \le\binom{n-1}{k-1}$ holds. Moreover, if $n\ge 2.5k-0.5$, then we have $gp(Kn_{n,k}) =\binom{n-1}{k-1}$, while if $2k+1\le n <2.5k-0.5$, then $gp(Kn_{n,k}) <\binom{n-1}{k-1}$ holds.
\end{theorem}


The proof of Theorem \ref{szlov} uses the following general result of Anand et al \cite{klavetal} that characterizes vertex subsets in general position.

\begin{theorem}[\cite{klavetal}]\label{klavko}
If $G$ is a connected graph, then a subset $S$ of the vertices of $G$ is in general position if and only if all the components $S_1, S_2, \dots, S_h$ of $G[S]$ are cliques in $G$ and
\begin{itemize}
    \item 
    for any $1\le i<j\le h$ and $s_i,s'_i\in S_i$, $s_j,s'_j \in S_j$ we have $d(s_i,s_j)=d(s'_i,s'_j)=:d(S_i,S_j)$ (where $d(x,y)$ denotes the distance of $x$ and $y$ in $G$),
    \item
    $d(S_i,S_j)\neq d(S_i,S_l)+d(S_l,S_j)$ for any $1\le i,j,l\le h$.
\end{itemize}
\end{theorem}

In the Kneser graph a clique corresponds to a family $\cF\subseteq \binom{[n]}{k}$ of pairwise disjoint sets and as there is no edge between different components, it follows that if $\cF_1,\cF_2,\dots, \cF_h$ correspond to the components of $G[S]$, then for any $F_i\in \cF_i$ and $F_j\in \cF_j$ with $i\neq j$ we have $F_i\cap F_j\neq \emptyset$. Families with this property are called \textit{cross-intersecting}.
 So the upper bound in Theorem \ref{kneser} will follow from the next result unless $n=2k+1$ in which case we will need some further reasonings.
 
\begin{theorem}\label{main}
Let $n\ge 2k+2$, $k\ge 4$ and let $\cF_1,\cF_2,\dots, \cF_h\subseteq \binom{[n]}{k}$ such that 
\begin{itemize}
    \item 
    $\cF_i\cap \cF_j=\emptyset$ for all $1\le i< j\le h$,
    \item
    $F_i\cap F'_i=\emptyset$ for all pairs of distinct sets $F_i,F'_i\in \cF_i$ for any $i=1,2,\dots,h$,
    \item
    $F_i\cap F_j\neq \emptyset$ for any $1\le i<j\le j$ and any $F_i\in \cF_i$, $F_j\in \cF_j$
\end{itemize}
hold. Then we have $\sum_{i=1}^h|\cF_i|\le \binom{n-1}{k-1}$.
\end{theorem}

Note that the first condition cannot be omitted as otherwise we could repeat some families that consist of a single set.

\bigskip

The remainder of the paper is organized as follows: Section 2 contains the proof of Theorem \ref{main} and in Section 3 we list some open problems along with some remarks.
\section{Proofs}

\begin{proof}[Proof of Theorem\ref{main}] Let $\cF_1,\cF_2,\dots,\cF_h\subseteq \binom{[n]}{k}$ satisfy the conditions of the theorem. As the $\cF_i$'s are families of pairwise disjoint sets, each of them are of size at most $n/k$ and we may assume that $|\cF_1|\le |\cF_2|\le \dots \le |\cF_h|=:t\le n/k$. If $t=1$, then $\cF=\cup_{i=1}^h\cF_i$ form an intersecting family and therefore by the celebrated theorem of Erd\H os, Ko and Rado \cite{ekr} we have $\sum_{i=1}^h|\cF_i|=h\le \binom{n-1}{k-1}$.

Suppose next that $t\ge 2$ holds. Then we claim $h\le \binom{n-1}{k-1}-\binom{n-k-1}{k-1}+1$. Indeed, let us fix one set $F_i$ from each $\cF_i$ for $i=1,2,\dots,h-1$ and two sets $F_h,F'_h\in \cF_h$. Then if 
\begin{itemize}
    \item 
    $|\cap_{i=1}^{h-1}F_i|\ge 2$, then $h-1\le \binom{n-2}{k-2}< \binom{n-1}{k-1}-\binom{n-k-1}{k-1}$,
    \item
    $\cap_{i=1}^{h-1}F_i$ consists of a single element $x$, then either $F_h$ or $F'_h$ cannot contain $x$ and as all $F_i$'s meet both $F_h$ and $F'_h$ we must have $h-1\le \binom{n-1}{k-1}-\binom{n-k-1}{k-1}$,
    \item
    $\cap_{i=1}^{h-1}F_i=\emptyset$, then $\{F_1,F_2,\dots, F_{h-1},F_h\}$ is intersecting with no common elements, so by a result of Hilton and Milner \cite{HM} we obtain $h\le \binom{n-1}{k-1}-\binom{n-k-1}{k-1}+1$.
\end{itemize}

Let $m_i$ denote the number of $j$'s such that $|\cF_j|\ge i$ holds. Then clearly we have  \begin{equation}\label{eq1}
    \sum_{i=1}^h|\cF_i|=h+\sum_{j=2}^tm_j\le h+\left(\frac{n}{k}-1\right)m_2.
\end{equation}

To bound $m_2$ we apply Bollob\'as's famous inequality \cite{b} that states that if $\{(A_1,B_1)\}_{i=1}^l$ are pairs of disjoint sets such that for any $1\le i\neq j\le l$ we have $A_i\cap B_j\neq \emptyset$, then $\sum_{i=1}^l\frac{1}{\binom{|A_i|+|B_i|}{|A_i|}}\le 1$ holds. For any $1\le i \le m_2$ we can pick two sets $F_i,G_i\in \cF_{h-m_2+i}$. Then we can define $2m_2$ pairs $\{(A_j,B_j)\}_{j=1}^{2m_2}$ such that for $1\le j\le m_2$ we have $A_j=F_j, B_j=G_j$ and $A_{2m_2-j}=G_j,B_{2m_2-j}=F_j$. As the $\cF_i$'s are cross-intersecting families of disjoint sets, therefore the pairs $\{(A_j,B_j)\}_{j=1}^{2m_2}$ satisfy the conditions of Bollob\'as's inequality and we obtain $\frac{2m_2}{\binom{2k}{k}}\le 1$ and thus $m_2\le \frac{1}{2}\binom{2k}{k}=\binom{2k-1}{k-1}$. Putting together (\ref{eq1}) and the bounds on $h$ and $m_2$ we obtain
$$
\sum_{i=1}^h|\cF_i|\le \binom{n-1}{k-1}-\binom{n-k-1}{k-1}+1+\frac{n-k}{k}\binom{2k-1}{k-1}.$$
Therefore it is enough to prove $\binom{n-k-1}{k-1}>\frac{n-k}{k}\binom{2k-1}{k-1}$. Observe that $$\frac{\binom{n-k}{k-1}}{\binom{n-k-1}{k-1}}=\frac{n-k}{n-2k+1}\ge \frac{n-k+1}{n-k}=\frac{\frac{n-k+1}{k}\binom{2k-1}{k-1}}{\frac{n-k}{k}\binom{2k-1}{k-1}},$$ therefore if $\binom{n_0-k-1}{k-1}>\frac{n_0-k}{k}\binom{2k-1}{k-1}$ holds for some $n_0$, then $\binom{n-k-1}{k-1}>\frac{n-k}{k}\binom{2k-1}{k-1}$ holds for $n\ge n_0$. Putting $n_0=3k+2$ the above inequality is equivalent to $$k\prod_{i=0}^{k-2}(2k+1-i)>(2k+2)\prod_{i=0}^{k-2}(2k-1-i)$$ which simpifies to $$k(2k+1)2k>(2k+2)(k+2)(k+1).$$
This holds for $k\ge 5$ and a similar calculation shows that if $k=4$, then the desired inequality holds if $n\ge 17=4k+1$.

In all missing cases, except for $k=4$, $n=16$, we have $n<4k$, therefore we have $m_j=0$ for all $j\ge 4$. So for the remaining pairs $n$ and $k$, we need to strengthen our bound on $m_2+m_3$ . We will need the following lemma, a slight generalization of Bollob\'as's result.

\begin{lemma}\label{bollgen}
Let $\{A_i,B_i\}_{i=1}^\alpha$ and $\{A_j,B_j,C_j\}_{j=\alpha+1}^\beta$ be pairs and triples of pairwise disjoint sets such that for any $1\le i<j\le \alpha+\beta$ we have $X_i\cap Y_j\neq \emptyset$ where $X$ and $Y$ can be any of $A,B$ and $C$. Then the following inequality holds:
$$
\sum_{i=1}^{\alpha+\beta}\frac{2}{\binom{|A_i|+|B_i|}{|A_i|}}+\sum_{j=1}^\beta\left(\frac{2}{\binom{|A_{\alpha+j}|+|C_{\alpha+j}|}{|A_{\alpha+j}|}}+\frac{2}{\binom{|B_{\alpha+j}|+|C_{\alpha+j}|}{|B_{\alpha+j}|}}-\frac{2}{\binom{|A_{\alpha+j}|+|B_{\alpha+j}|+|C_{\alpha+j}|}{|A_{\alpha+j}|}}-\frac{2}{\binom{|A_{\alpha+j}|+|B_{\alpha+j}|+|C_{\alpha+j}|}{|B_{\alpha+j}|}}\right)\le 1$$
\end{lemma}

\begin{proof}
Let us define $M$ to be $\bigcup_{i=1}^\alpha(A_i\cup B_i)\cup \bigcup_{j=1}^\beta(A_{\alpha+j}\cup B_{\alpha+j}\cup C_{\alpha+j})$ and let us write $|M|=m$. Just as before, let us introduce a family $\{S_i,T_i\}_{i=1}^{2(\alpha+\beta)}$ of disjoint pairs as $S_i=A_i, T_i=B_i$ and $S_{2(\alpha+\beta)-j}=B_j, T_{2(\alpha+\beta)-j}=A_j$ for all $1\le i,j\le \alpha+\beta$. We count the pairs $(\pi,j)$ such that $\pi$ is a permutation of the elements of $M$ and $1\le j\le 2(\alpha+\beta)$ with all elements of $S_j$ preceding all elements of $T_j$ in $\pi$ that is $\max\{\pi^{i-1}(s):s\in S_j\}<\min\{\pi^{-1}(t):t\in T_j\}$. We denote this by $S_j<_\pi T_j$. For every fixed $j$ there exist exactly $|S_j|!|T_j|!(m-|S_j|-|T_j|)!\binom{m}{|S_j|+|T_j|}$ permutations $\pi$ with $S_j<_{\pi} T_j$. On the other hand for any fixed $\pi$ there exists at most one $j$ with $S_j<_{\pi} T_j$. Indeed, if $i\neq j,2(\alpha+\beta)-j$, then both $S_i$ and $T_i$ meet both $S_j$ and $T_j$, while clearly if $S_j<_{\pi} T_j$, then $S_{2(\alpha+\beta)-j}=T_j\not<_{\pi} S_j=T_{2(\alpha+\beta)-j}$. These observations would yield Bollob\'as's original inequality, but we haven't used the existence of the $C_j$'s. Observe that if $A_j<_{\pi} C_j$, $C_j<_{\pi} A_j$, $B_j<_{\pi} C_j$ or $C_j<_{\pi} B_j$, then again by the cross-intersecting property $(\pi,i)$ can be a pair counted only if $i=j$ or $i=2(\alpha+\beta)-j$ and at least one of $A_i<_{\pi} B_i \cup C_i$, $B_i\cup C_i<_{\pi} A_i$, $C_i\cup B_i <_{\pi} A_i$, $C_i\cup A_i<_{\pi} B_i$ holds. Counting $j$ and $2(\alpha+\beta)-j$ cases together this yields
$$    \sum_{j=1}^{\alpha+\beta}2|A_j|!|B_j|!(m-|A_j|-|B_j|)!\binom{m}{|A_j|+|B_j|}\le m!$$
$$-\sum_{j=1}^{\alpha+\beta}2\left[|A_j|!|C_j|!(m-|A_j|-|C_j|)!\binom{m}{|A_j|+|C_j|}+|B_j|!|C_j|!(m-|C_j|-|B_j|)!\binom{m}{|C_j|+|B_j|}\right]
$$
$$+\sum_{j=1}^{\beta}2|A_{\alpha+j}|!(|B_{\alpha+j}|+|C_{\alpha+j}|)!(m-|A_{\alpha+j}|-|B_{\alpha+j}|-|C_{\alpha+j}|)!\binom{m}{|A_{\alpha+j}|+|B_{\alpha+j}|+|C_{\alpha+j}|}$$
$$+\sum_{j=1}^{\beta}2|B_{\alpha+j}|!(|A_{\alpha+j}|+|C_{\alpha+j}|)!(m-|A_{\alpha+j}|-|B_{\alpha+j}|-|C_{\alpha+j}|)!\binom{m}{|A_{\alpha+j}|+|B_{\alpha+j}|+|C_{\alpha+j}|}$$
Dividing by $m!$ and rearranging yields the statement of the lemma.
\end{proof}

We apply Lemma \ref{bollgen} to the families $\cF_{h-m_2+1},\dots,\cF_h$ with $\beta=m_3$ and $\alpha=m_2-m_3$. As all sets in the $\cF_i$'s are of size $k$ we obtain
\begin{equation}\label{eq3}
    \frac{2(m_2-m_3)}{\binom{2k}{k}}+\frac{6m_3}{\binom{2k}{k}}-\frac{6m_3}{\binom{3k}{k}}\le 1.
\end{equation}
As $\binom{3k}{k}\ge 3\binom{2k}{k}$ for $k\ge 3$, the left hand side of the above equation is greater than $\frac{2(m_2-m_3)}{\binom{2k}{k}}+\frac{4m_3}{\binom{2k}{k}}=\frac{2(m_2+m_3)}{\binom{2k}{k}}$. Therefore we obtain $m_2+m_3\le \frac{1}{2}\binom{2k}{k}=\binom{2k-1}{k-1}$. So for $n<4k$ we have the bound
\begin{equation}\label{eq2}
\sum_{i=1}^h|\cF_i|\le h+m_2+m_3\le \binom{n-1}{k-1}-\binom{n-k-1}{k-1}+1+\binom{2k-1}{k-1}.    
\end{equation}
Suppose first that $n\ge 3k$ holds. Plugging into (\ref{eq2}) we obtain the upper bound $\binom{n-1}{k-1}+1$. To get rid of the extra 1, we need to use the uniqueness part of the Hilton-Milner theorem \cite{HM} that we used to get our bound on $h$. It states that if $k\ge 4$ and an intersecting family $\cF\subseteq \binom{[n]}{k}$ with $\cap_{F\in \cF}F=\emptyset$ has size $\binom{n-1}{k-1}-\binom{n-k-1}{k-1}+1$, then there exists $x\in [n]$ and $x\notin G\subseteq [n]$ such that $\cF=\{G\}\cup \{F: x\in F, F\cap G\neq \emptyset\}$. Observe that for any $H\neq G$ with $x\notin H$ there exist lots of sets $F\in \cF$ that are disjoint with $H$, so only sets $H'$ that contain $x$ can be added to the $\cF_j$'s. But as all $\cF_j$'s consist of pairwise disjoint sets, such an $H'$ can only be added to the $\cF_j$ containing $G$. Also, at most one such set can be added as again this $\cF_j$ consists of pairwise disjoint sets. We obtained that if $t\ge 2$ and $h=\binom{n-1}{k-1}-\binom{n-k-1}{k-1}+1$, then $\sum_{j=1}^h|\cF_j|\le \binom{n-1}{k-1}-\binom{n-k-1}{k-1}+2<\binom{n-1}{k-1}$. 

Next, we assume that $2k+2\le n<3k$. Then we have $t\le 2$ and therefore the family $\cF':=\cup_{i=1}^h\cF_i$ has the property that for any $F\in \cF'$ there exists at most one other $G\in \cF'$ that is disjoint with $F$. Such families are called \textit{$(\le 1)$-almost intersecting} and Gerbner et al. \cite{Getal} proved that whenever $2k+2\le n$ holds, then any $(\le 1)$-almost intersecting family $\cG\subseteq \binom{[n]}{k}$ has size at most $\binom{n-1}{k-1}$.

Finally, if $n=16$, $k=4$, then we need to bound $h+m_2+m_3+m_4\le h+m_2+2m_3\le h+2m_2+3m_3$. As $\binom{3k}{k}=\binom{12}{4}>6\binom{8}{4}=\binom{2k}{k}$, (\ref{eq3}) implies $2m_2+3m_3\le \binom{8}{4}$. Using the Hilton-Milner bound $h\le \binom{n-1}{k-1}-\binom{n-k-1}{k-1}+1$ and plugging in $n=16$, we obtain $\sum_{i=1}^h|\cF_i|\le h+2m_2+3m_3\le \binom{n-1}{k-1}-\binom{11}{3}+1+\binom{8}{4}<\binom{n-1}{k-1}$.  This concludes the proof.
\end{proof}

\begin{proof}[Proof of Theorem \ref{kneser}]
Theorem \ref{main} shows that $Kn_{n,k}\le \binom{n-1}{k-1}$ holds if $n\ge 2k+2$. Observe that $diam(Kn_{n,k})\le 3$ if and only if $n\ge 2.5k-0.5$. Also, Theorem \ref{klavko} yields that if the diameter of a graph $G$ is at most 3, then any independent set in $G$ is in general position. The largest independent sets in $Kn_{n,k}$ correspond to \textit{stars}, i.e. families $\cS_x=\{H \in \binom{[n]}{k}:x \in H\}$ for some $x\in [n]$. Therefore, $gp(Kn_{n,k})\ge \binom{n-1}{k-1}$ holds provided $n\ge 2.5k-0.5$.

If $2k+2\le n<2.5k-0.5$, then the upper bound  of Theorem \ref{main} is based on the result of Gerbner et al \cite{Getal} on $(\le 1)$-almost intersecting families. Their result also states that the only $(\le 1)$-almost intersecting families of size $\binom{n-1}{k-1}$ are stars. But if $n<2.5k-0.5$, then $\{ H\in \binom{[n]}{k}:1 \in H\}$ is not in general position as shown by the following example:
let $n=2k+M$ with $1\le M< 0.5k-0.5$ and $F_1=[k]$, $F_2=\{1,2,\dots,k-M-1\}\cup \{k+1,k+2,\dots,k+M+1\}$. We claim that $d_{Kn_{n,k}}(F_1,F_2)\ge 4$. Indeed, as $C:=[n]\setminus (F_1\cup F_2)$ is of size $k-1$, we have $d_{Kn_{n,k}}(F_1,F_2)\ge 3$. Suppose $G_1,G_2$ are $k$-subsets of $[n]$ with $F_1\cap G_1=G_1\cap G_2=\emptyset$. Let us define $\ell=|G_1\cap F_2|$. As $G_1$ is disjoint with $F_1$, so with $F_1 \cap F_2$, we have $\ell\le M+1$. Therefore $|C\cap G_1|\ge k-M-1$ must hold. As $G_2$ is disjoint with $G_1$, we obtain $|C\cap G_2|\le M$, but as $|F_1\setminus F_2|=M+1$ and $2M+1<k$, $G_2$ must meet $F_2$, so indeed $d_{Kn_{n,k}}(F_1,F_2)\ge 4$ holds. On the other hand, for any $x\in F_2\setminus F_1$ and $y,z\in F_1\setminus F_2$, the sets $F_1, C\cup \{x\}, F_2\setminus \{x\}\cup \{y\}, C\cup \{z\},F_2$ form a path of length 4, therefore a geodesic with $1\in F_2\setminus \{x\}\cup \{z\}$. This shows that $\{ H\in \binom{[n]}{k}:1 \in H\}$ is not in general position.
Therefore if $2k+2 \le n<2.5k-0.5$ holds, then we have $gp(Kn_{n,k})<\binom{n-1}{k-1}$.

Finally, let us consider the case $n=2k+1$. Again, vertices corresponding to sets of stars are not in general position and all other independent sets have size smaller than $\binom{n-1}{k-1}$. So suppose $F,F'$ are disjoint sets in a family $\cF$ corresponding to vertices in general position. Then by Theorem \ref{klavko}, for any set $G\neq F,F'$ in $\cF$ we must have $d(G,F)=d(G,F')$. Observe that in $Kn_{2k+1,k}$ we have $d(H,H')=\min\{2(k-|H\cap H'|),2|H\cap H'|+1\}$. 

Let us first assume that $k=2l+1$ is odd. Then by the above, for any $G\in \cF$ we must have $|G\cap F|=|G\cap F'|=l$ and the unique element $x \in [2k+1]\setminus (F\cup F')$ must belong to $G$. Therefore, with the notation of the proof of Theorem \ref{main}, we have $m_2=1$ and $h \le \binom{n-1}{k-1}-\binom{n-k-1}{k-1}+1$ and thus $|\cF|\le \binom{n-1}{k-1}-\binom{n-k-1}{k-1}+2<\binom{n-1}{k-1}$.

Let us assume that $k=2l$ is even. Then by the above, for any $G\neq F,F'$ in $\cF$ we must have $|G\cap F|=|G\cap F'|=l$ and thus $G \subseteq F\cup F'$. If we take one set from each disjoint pair, we obtain a family $\cG\subseteq \binom{[2k]}{k}$ such that any pairwise intersection is of the same size. By Fisher's inequality, we obtain that the number $m_2$ of pairs is at most $2k$. Moreover, as all sets of $\cF$ are $k$-subsets of $[2k]$, we must have $h\le \frac{1}{2}\binom{2k}{k}$. Therefore, we need to show $\frac{1}{2}\binom{2k}{k}+2k<\binom{2k}{k-1}=\binom{2k}{k}\frac{k}{k+1}$ which is equivalent to $\frac{2k(2k+2)}{k-1}<\binom{2k}{k}$. This holds for $k\ge 4$.
\end{proof}

\section{Concluding remarks}

First of all, it remains an open problem to determine $gp(Kn_{n,k})$ for $2k+1\le n <2.5k-0.5$.

Let us finish this short note with two remarks. First observe that an $(\le 1)$-almost intersecting family $\cF\subseteq \binom{[n]}{k}$ corresponds to a subset $U$ of the vertices of $Kn_{n,k}$ such that $Kn_{n,k}[U]$ does not contain a path on three vertices. There have been recent developments \cite{AT,Getal2,T} in the general problem of finding the largest possible size of a subset $U$ of the vertices of $Kn_{n,k}$ such that $Kn_{n,k}[U]$ does not contain some fixed forbidden graph $F$. Note that independently of the host graph $G$, if a subset $S$ of the vertices of $G$ is in general position, then $G[S]$ cannot contain a path on three vertices as an \textit{induced} subgraph. Returning to the Kneser graph $Kn_{n,k}$ it would be interesting to address the induced version of the vertex Tur\'an problems mentioned above.

There have been lots of applications and generalizations of Bollob\'as's inequality. Very recently O'Neill and Verstra\"ete \cite{OV} obtained Bollob\'as type results for $k$-tuples. Their condition to generalize disjoint pairs is completely different from the condition of Lemma \ref{bollgen}. More importantly pairwise disjoint, cross-intersecting families were introduced by R\'enyi \cite{R} as \textit{qualitatively independent partitions} if the extra condition that $\cup_{F\in\cF_i}F=[n]$ holds for all $1\le i \le h$ is added, and the uniformity condition $|F|=k$ for all $F\in \cup_{i=1}^h\cF_i$ is replaced by $|\cF_i|=d$ for all $1\le i\le h$. Gargano, K\"orner and Vaccaro proved \cite{GKV} that for any fixed $d\ge 2$ as $n$ tends to infinity the maximum number of qualitatively independent $d$-partitions is $2^{(\frac{2}{d}-o(1))n}$. Based on their construction, for any fixed $d$ one can obtain $2^{(2-o(1))k}$ many pairwise disjoint cross-intersecting $d$-tuples of $k$-sets as $k$ tends to infinity.
\bigskip

\textbf{Acknowledgement}: I would like to thank Sandi Klav\v zar and Gregor Rus for pointing out that sets of the star are not in generalposition if $n<2.5k-0.5$. I would like to thank M\'at\'e Vizer for showing me the relation of the families considered in the paper to qualitative independent partitions and G\'abor Simonyi for providing me a short introduction to this topic.

\textbf{Funding}: Research supported by the National Research, Development and Innovation Office - NKFIH under the grants SNN 129364 and K 116769.

\end{document}